\documentclass{amsart}

\usepackage{graphicx}
\usepackage{mathtools}
\usepackage{tikz}
\usetikzlibrary{positioning}
\usepackage{mathabx}
\usepackage{enumitem}
\usepackage[colorinlistoftodos,backgroundcolor=black!10,linecolor=black]{todonotes}
\usepackage{amssymb}
\usepackage{float}

\usepackage{amsrefs}

\newtheorem{theorem}{Theorem}[section]
\newtheorem{lemma}[theorem]{Lemma}
\newtheorem{proposition}[theorem]{Proposition}
\newtheorem{conjecture}[theorem]{Conjecture}
\newtheorem{corollary}[theorem]{Corollary}

\theoremstyle{definition}
\newtheorem{definition}[theorem]{Definition}

\theoremstyle{remark}

\numberwithin{equation}{section}

\newcommand{\QQ}{\mathbb{Q}}
\newcommand{\ZZ}{\mathbb{Z}}
\newcommand{\CC}{\mathbb{C}}

\newcommand{\FF}{\mathbb{F}}

\newcommand{\Gal}{\textrm{Gal}}

\newcommand{\mfp}{\mathfrak{p}}

\newcommand{\mfq}{\mathfrak{q}}

\newcommand{\rk}{\textrm{rank }}

\newcommand{\vare}{\varepsilon}

\newcommand{\ClLp}{{A}}

\newcommand{\mmod}{\text{ mod }}

\begin{document}

\title{The unit norm index and $p$-class group in certain degree $\ell$ extensions}

\author{Ariella Kirsch}

\date{\today}

\begin{abstract}
We examine when units in a field are the norms of elements in an extension
field, given certain conditions. We apply these results to the study of
the $\ell$-class groups in lifts of the anti-cyclotomic $\ZZ_2$-extension
of $\QQ(i)$.
\end{abstract}

\maketitle

\section{Introduction}\label{sec:intro}

In \cite{iwasawa1973mu}, Iwasawa
showed that there exist $\ZZ_p$-extensions in which the $\mu$-invariant is
non-zero. In \cite{washington1975class}, Washington showed that there are
$\ZZ_p$-extensions in which the $\ell$-part of the class group is unbounded,
where $\ell \neq p$ is prime.
Both of these results began with a $\ZZ_p$-extension $K_\infty/K_0$
and then lifted it via a cyclic extension $L_0/K_0$ to a $\ZZ_p$-extension
$L_\infty/L_0$ where $L_\infty = K_\infty L_0$.
 Then Chevalley's Formula (see \ref{chevalley})
was applied to the intermediate fields $L_n/K_n$.

Chevalley's Formula relies on the unit index
$$[E_{K_n}:E_{K_n} \cap N_{L_n/K_n}(L_n^\times)]$$
where $L_n/K_n$ is a cyclic extension, $E_{K_n}$ is the unit group of $K_n$ and
$N_{L_n/K_n}$ is the norm map from $L_n$ to $K_n$.
This index is in general difficult to compute, but fortunately trivial estimates
were sufficient for the results above. In this paper,
inspired by the work of Wittmann \cite{wittmann2004} and Gerth \cite{gerth1976},
we study this index
in the case where $L_0/K_0$ is a degree $\ell$ extension ($\ell \neq 2$)
and $K_\infty/K_0$ is the anti-cyclotomic $\ZZ_2$-extension of an imaginary
quadratic field.

We show that given certain
conditions on $\ell$ and the ramified primes in $L_n/K_n$, either all of the units in $K_n$
are norms of elements in $L_n$ modulo $\ell^{th}$ powers or none of them are.
We then give heuristics that indicate that in general,
the difference between the actual value of the index and the trivial estimate
is probably bounded.
Therefore the principal cause of an unbounded $\ell$-class number in $L_\infty/L_0$
should be due to the contribution from Chevalley's Formula and not from some other
unexplained phenomenon.

\section{The anti-cyclotomic $\ZZ_2$-extension}\label{sec:Kn}

Let $K_0$ be an imaginary quadratic field.
Let $K_n$ be the $n^{th}$ layer, $n \geq 1$, of the anti-cyclotomic $\ZZ_2$-extension
of $K_0$.
We fix $\ell$ an odd prime. If $K_0 = \QQ(\sqrt{-3})$, then we require that $\ell \neq 3$.
This ensures that $K_n$ does not contain $\ell^{th}$ roots of unity.

In this section, we present results on the units and primes in $K_n$,
which we will then use to prove our main result, Theorem \ref{allornothing}.

\subsection{The Units of $K_n$}

Since $K_n$ is a totally imaginary field of degree $2^{n+1}$ over $\QQ$, it
has no real embeddings ($r_1 = 0$) and $2^n$ pairs of
complex embeddings ($r_2 = 2^n$). By Dirichlet's unit theorem, $K_n$ has $$r_1+r_2 - 1 = 2^n-1$$
fundamental units.

Let $E_n$ be the unit group of $K_n$ and let $U_n$ be the units of $K_n$ modulo $\ell^{th}$ powers.

We also have $\Gal(K_n/\QQ) \simeq D_n$. Let $\sigma$ and $\tau$ be generators of the Galois group where
$\sigma$ has order $2^n$ and $\tau$ has order 2.
Let $\tau$ be complex conjugation under some embedding $K_n \hookrightarrow \CC$.
We then have the following relation: $$\tau \sigma^i = \sigma^{-i} \tau.$$

\begin{proposition}\label{one_unit_fixed}
At least one unit of $K_n$ that is not a unit in $K_{n-1}$ is fixed by $\tau$:
$$\exists \hspace{.1in} u \in E_{K_n} \big \backslash E_{K_{n-1}} \text{ such that } \tau u = u.$$
\end{proposition}
\begin{proof}
Let $F_n$ be the fixed field of $\tau$.
This field is of degree $2^n$ over $\QQ$.
Consider the embeddings $\sigma^i$ applied to $F_n$.
Since $\tau$ fixes $F_n$, we now determine which of the Galois elements
stabilize $F_n$,
since these will be the real embeddings. Since $$\tau \sigma^i (F_n) = \sigma^{-i} \tau (F_n) = \sigma^{-i}(F_n),$$
it's sufficient to consider just the elements of the form $\sigma^i$.
Let $x \in F_n$. Say $\sigma^i$ stabilizes $F_n$ and therefore $\tau \sigma^i (x) = \sigma^i(x)$.
Then $$\sigma^i(x) = \tau \sigma^i (x) = \sigma^{-i} \tau (x) = \sigma^{-i} (x)
\implies \sigma^{2i}(x) = x.$$
Therefore $\sigma^{2i} \in \Gal(K_n/F_n)$. Since $\sigma^{2i} \neq \tau$, we have
$i=0$ or $i=2^{n-1}$ and so there are exactly
two real embeddings. Therefore there are $2^{n-1}-1$ pairs of complex embeddings and so
$F_n$ has
$$r_1 + r_2 -1 = 2 + 2^{n-1}-1 -1 = 2^{n-1}$$ fundamental units.

The units of $K_{n-1}$ are embedded into the units of $K_n$, and similarly the units of $F_n$ are also embedded into the units of $K_n$. Since $F_n$ has $2^{n-1}$ independent units and $K_{n-1}$ has $2^{n-1}-1$ independent units, this means that at least one unit of $K_n$ that is not a unit in $K_{n-1}$ is fixed by $\tau$.
\end{proof}

Recall that we use $U_n$ to represent the units of $K_n$ modulo $\ell^{th}$ powers.
We now need to understand the structure of the relative units of $K_n$ modulo $\ell^{th}$ powers, which we define as
$$U_n^{rel} = \ker\left(N_{K_n/K_{n-1}}: U_n \rightarrow U_{n-1} \right).$$

\begin{corollary}\label{relative_unit_fixed}
At least one non-trivial relative unit of $K_n$ is fixed by $\tau$.
\end{corollary}
\begin{proof}
Let $u \in K_n$ be a unit not in $K_{n-1}$ that is fixed by $\tau$ (Proposition
\ref{one_unit_fixed}). For ease of notation, let $\alpha = \sigma^{2^{n-1}}$
and let $v=u^{\alpha-1}$. Since $N_{K_n/K_{n-1}}=\alpha+1$, $v$ is a relative unit.
Since $\tau$ and $\alpha$ commute,
$$\tau v = \tau (u^{\alpha-1}) = (\tau u)^{\alpha-1} = u^{\alpha-1} = v.$$
Therefore $v$ is a relative unit that is fixed by $\tau$.

We now need to ensure that the relative unit is non-trivial modulo $\ell^{th}$ powers.
If $v$ is trivial, then $v = w^{\ell^m}$ for some unit  $w \in K_n$ and some $m > 0$,
where we assume $m$ is maximal.
Then $$(\tau w)^{\ell^m} = \tau v = v = w^{\ell^m}.$$
Since $K_n$ does not contain $\ell^{th}$ roots of unity (since
we excluded $\ell=3$ when $K_0 = \QQ(\sqrt{-3})$), we have $\tau w = w$. Because $m$
was maximal, $w$ is not an $\ell^{th}$ power and therefore $w$ is non-trivial in $U_n$.
Since $$(w^{\alpha+1})^{\ell^m} = v^{\alpha+1} = 1$$ and $K_n$ does not contain the
$\ell^{th}$ roots of unity, $w^{\alpha+1}=1$ and so we have found a non-trivial relative
unit that is fixed by $\tau$.
\end{proof}

We also need the following result on the structure of the relative unit group.
\begin{proposition}\label{gens_relative_units}
There exists $u \in U_n$ such that the
relative units modulo $\ell^{th}$ powers, $U_n^{rel}$, are spanned by
$$\{u, \hspace{.1in} \sigma u, \hspace{.1in} \sigma^2 u,  \hspace{.1in} \ldots, \hspace{.1in} \sigma ^{2^{n-1}-1} u\}$$
and furthermore
$$\sigma^{2^{n-1}} u \equiv u^{-1} \mmod \ell^{th} \text{ powers}.$$
\end{proposition}
\begin{proof}
The Galois group $\Gal(K_n/\QQ)$ has generators $\sigma$ and $\tau$, where
$\tau$ is complex conjugation under some embedding. Then $(\sigma^j, \tau \sigma^j)$ are the
pairs of complex conjugate embeddings of $K_n$ into $\CC$ for $0 \leq j < 2^n$.
By \cite{washington2012introduction}*{Lemma 5.27}, there exists a unit $\vare \in E_n$ such that
$$\{\vare, \hspace{.1in} \sigma \vare, \hspace{.1in} \sigma^2 \vare, \ldots, \hspace{.1in} \sigma ^{2^{n}-2} \vare\}$$
generates a subgroup $H$ of finite index in $E_n$. Therefore, since
$$\vare^{1 + \sigma + \ldots + \sigma^{2^n-1}} \in \{ \pm 1, \pm i, \pm \zeta_3 \},$$
we have
$$E_n \otimes_\ZZ \QQ \simeq \QQ[\langle \sigma \rangle] \big/ (1+\sigma+\ldots+\sigma^{2^n-1}).$$

The units modulo $\ell^{th}$ powers are $$U_n = E_n \big/ E_n^\ell \simeq E_n \otimes_\ZZ \FF_\ell.$$ Let $g \in \langle \sigma \rangle$.
Since $H$ is $g$-stable, we can calculate the characteristic polynomial $f_g(x)$ of $g$ using $H$. Since $H$ is a
$\ZZ$-module, $f_g(x) \in \ZZ[x]$. We can also compute $f_g(x)$ using a basis for $E_n$ modulo torsion,
and since $H$ and $E_n$ span $E_n \otimes \QQ$, they yield the same characteristic polynomials. We can reduce
$f_g(x)$ modulo $\ell$ to get the characteristic polynomial of $g$ on $U_n$ and on $H \otimes_\ZZ \FF_\ell$.

The Brauer-Nesbitt Theorem says that the semi-simplification is determined by the characteristic polynomials
of $g \in \langle \sigma \rangle$. Observe that $| \langle \sigma \rangle | = 2^n$ is relatively prime to $\ell$,
the characteristic of $\FF_\ell$. Therefore the representations of $\langle \sigma \rangle$ on $U_n$ and on
$H \otimes_\ZZ \FF_\ell$ are already semi-simple and so these representations are isomorphic:
\begin{align}\label{brauernesbittiso}
U_n \simeq H \otimes_\ZZ \FF_\ell \simeq \FF_\ell[\langle \sigma \rangle] \Big / (1+\sigma+\ldots+\sigma^{2^n-1}).
\end{align}

Now let $\alpha = \sigma^{2^{n-1}}$. The norm map is $$N_{K_n/K_{n-1}} = 1 + \alpha.$$
Let $u \in U_n^{rel}$. Then $$u^2 = u^{1+\alpha} u^{1-\alpha} = u^{1-\alpha}$$
since $u^{1+\alpha}=1$. Since $U$ has odd order ($\ell \neq 2$), this implies
that $U_n^{rel}$ is contained in $(U_n^{rel})^{1-\alpha} \subseteq (U_n)^{1-\alpha}$.
Now consider $u' \in U_n$. Then $$((u')^{1-\alpha})^{1+\alpha} = (u')^{1-\alpha^2} = 1$$
and therefore $U_n^{1-\alpha} \subseteq U_n^{rel}$. Thus $U_n^{rel} = U_n^{1-\alpha}$.
Let $u \in U_n^{rel}$ correspond to $1-\alpha$ by the isomorphism in Equation
\ref{brauernesbittiso}.
Therefore $U_n^{rel}$ is spanned by
$$\{u, \hspace{.1in} \sigma u, \hspace{.1in} \sigma^2 u, \ldots, \hspace{.1in} \sigma ^{2^{n-1}-1} u\}$$
and
$$\sigma^{2^{n-1}} u \equiv u^{-1} \mmod \ell^{th} \text{ powers}.$$
\end{proof}

Finally, we have the following proposition
showing that the only `new' units introduced at the $n^{th}$ layer are the relative units.
\begin{proposition}\label{unitsdirectsum}
$U_n \simeq U_n^{rel} \oplus U_{n-1}$ as $\Gal(K_n/\QQ)$-modules.
\end{proposition}
\begin{proof}
As in the previous proof, let $\alpha = \sigma^{2^{n-1}}$.
Then we can see that $U_n$ is generated by $U_n^{rel}$ and $U_n^{1+\alpha}$,
since $u^2 = u^{1-\alpha} u^{1+\alpha}$ and $u^{1-\alpha} \in U_n^{rel}$.
Furthermore, if \mbox{$u \in U_n^{rel} \cap U_{n}^{1+\alpha}$}, then $u = v^{1+\alpha}$
for some $v \in U_n$ and $N_{K_n/K_{n-1}} u = u^{1+\alpha} = 1.$ Then
    $$1 = u^{1+\alpha} = v^{2+2\alpha} = u^2$$ which implies that $u=1$.
Therefore we can prove the proposition by showing that $U_n^{1+\alpha} \simeq U_{n-1}$.

Recall that $U_n$ is the set of units modulo $\ell^{th}$ powers:
$$U_n = E_n / E_n^\ell.$$
We wish to show that $( E_n / E_n^\ell )^{1+\alpha} \simeq (E_{n-1} / E_{n-1}^\ell)$.
We claim that this is an isomorphism via the map
$$\psi: x^{1+\alpha} \mod E_n^\ell \mapsto x^{1+\alpha} \mod E_{n-1}^\ell.$$

First, we show that the map is well-defined. If $x^{1+\alpha} = v^\ell$ for some $v \in E_n$, then
$$(v^\ell)^\alpha = (x^{1+\alpha})^\alpha = x^{1+\alpha} = v^\ell$$
and therefore $v^\ell$ is fixed by $\alpha$. Since $K_n$ does not contain any non-trivial
$\ell^{th}$ roots of unity, this means that $v$ must be fixed by $\alpha$, and so
$v \in E_{n-1}$.

Now we show that the map is surjective. If $y \in E_{n-1}$, then
$\psi(y^{1+\alpha}) = y^2 \mod E_{n-1}^\ell$. Since $\ell$ is odd, the map is surjective.

Finally we show injectivity. Let $x^{1+\alpha} = w^\ell$ for some $w \in E_{n-1}$.
Then $$(x^{1+\alpha})^{1+\alpha} = (x^{1+\alpha})^2$$ and so $(x^2/w^\ell)^{1+\alpha} = 1$.
This means that the square of the class of $x^{1+\alpha} \mod E_n$ is trivial, and therefore
the class itself must be trivial. Therefore $\psi$ is an isomorphism.
\end{proof}

\subsection{Galois Structure of the Primes of $K_n$}

We need to understand the behavior of the primes in the anti-cyclotomic extension.
\begin{lemma}[Hubbard-Washington {\cite{hubbardwashington2017}*{Lemma 1}}] \label{splitcompletely}
Let $K_0$ be an imaginary quadratic field and
let $K_\infty/K_0$ be the anti-cyclotomic $\ZZ_p$-extension of $K_0$.
If a prime $q \neq p$ is inert in $K_0/\QQ$ then $q$ splits completely in $K_\infty/K_0$.
\end{lemma}

We also require the following lemma relating the $\sigma$ action on the primes of $K_n$
to the $\tau$ action.

\begin{lemma}\label{fixedprimeprop}
Let $\mfp, \sigma \mfp, \ldots, \sigma^{2^n-1} \mfp$ be the primes
in $K_n$ lying over a rational prime $p$ that is inert in $K_0/\QQ$.
Let $\tau \mfp = \sigma^k \mfp$. Then $k$ is even if and only if there
is at least one prime lying over $p$ that is fixed by $\tau$.
\end{lemma}
\begin{proof}
First, assume $k$ is even: $k = 2k'$. Then
$$\tau \sigma^{k'} \mfp = \sigma^{-k'} \tau \mfp = \sigma^{-k'+2k'} \mfp = \sigma^{k'} \mfp$$
and therefore $\sigma^{k'} \mfp$ is fixed by $\tau$.

Now assume $k$ is odd: $k = 2k'+1$. Then assume $\sigma^j \mfp$ is fixed by $\tau$: $\tau \sigma^j \mfp = \sigma^j \mfp$. But
$$\sigma^j \mfp = \tau \sigma^j \mfp = \sigma^{-j} \tau \mfp = \sigma^{-j+2k'+1} \mfp$$
which implies $$j \equiv -j+2k'+1 \mmod 2^n \implies 2k'+1 \equiv 2j \mmod 2^n.$$ This is a contradiction.
\end{proof}

From this point forward, we let our imaginary quadratic field $K_0$ be $\QQ(i)$
in order to simplify the already technical proofs. This choice has several important
consequences.
We will be interested in primes that are inert in $K_0$ (so that they split
completely in $K_\infty/K_0$). For other base fields, we would have different
congruence conditions determining which primes ramify in $L/\QQ$. We would also have
different congruence conditions determining which primes are fixed by elements of the
Galois group (see Proposition \ref{7fixprime19}). Finally, we will also see that the
choice of $K_0=\QQ(i)$ means that $K_1$ is a cyclotomic field; this is not true in general.

\begin{proposition}\label{7fixprime19}
Let $K_0 = \QQ(i)$.
Let $K_\infty/K_0$ be the anti-cyclotomic $\ZZ_2$-extension.
Let
$$\Gal(K_n/\QQ) \simeq D_n = \langle \sigma, \tau \rangle$$ where $\sigma$ has order $2^n$ and $\tau$ has order 2
and restricts to the generator of $\Gal(K_0/\QQ)$.
If $p \equiv 3 \text{ mod } 8$, then none of the primes above $p$ in $K_n$ are fixed by $\tau$.
If $p \equiv 7 \text{ mod } 8$, then at least one of the primes above $p$ in $K_n$ is fixed by $\tau$.
\end{proposition}
\begin{proof}
For the anti-cyclotomic extension, $K_1 = \QQ(\zeta_8)$.  Let $F_n$ be the
fixed field of $\tau$. Then $F_1 = \QQ(\sqrt{2})$.
We are only interested in primes that are $3 \text{ mod } 4$, since those are the primes
that are inert in $\QQ(i)/\QQ$.

We have the following diagram of fields:

\begin{center}
\begin{tikzpicture}[node distance = 1.75cm, auto]
    \node (Q) at (0,0) {$\QQ$};
    \node (K0) at (2,.5) {$K_0$};
    \node (F1) at (0,1.5) {$F_1$};
    \node (K1) at (2,2) {$K_1$};
    \node (Fn) at (0,3.5) {$F_n$};
    \node (Kn) at (2,4) {$K_n$};
    \draw[-] (Q) to node [swap] {2} (K0);
    \draw[-] (F1) to node [swap] {2} (K1);
    \draw[-] (Fn) to node [swap] {2} (Kn);
    \draw[-] (Q) to node [swap] {2} (F1);
    \draw[-] (F1) to node [swap] {$2^{n-1}$} (Fn);
    \draw[-] (K0) to node [swap] {2} (K1);
    \draw[-] (K1) to node [swap] {$2^{n-1}$} (Kn);
\end{tikzpicture}
\end{center}
By Lemma \ref{splitcompletely}, primes that are inert in $K_0$ split
completely in $K_n/K_0$.
Primes that are $3 \mmod 8$ are inert in $F_1/\QQ$ and therefore must split completely
in $K_n/F_1$. This means there can be no primes in $K_n$ which are fixed by $\tau$.

Now let $\mfp \in K_n$, $n>1$, be a prime ideal lying over $p \in \QQ$, $p \equiv 7 \mmod 8$.
Then
$$N_{K_n/K_1} \mfp = \mfp^{1 + \sigma^2 + \sigma^4 + \ldots + \sigma^{2^n-2}} \coloneqq \mfq.$$
Since $p \equiv 7 \mmod 8$, $p$ splits in $F_1/\QQ$ and is inert in $K_1/F_1$.
Therefore $\tau \mfq = \mfq$.
Therefore $\tau \mfp  = \sigma^{2j} \mfp$ for some $j$ and so by Lemma
 \ref{fixedprimeprop}, there must be a prime in $K_n$ lying over $p$
that is fixed by $\tau$.
\end{proof}

\section{Cyclic Extensions of the Anti-Cyclotomic Extension}\label{sec:Ln}

Let $L$ be a cyclic degree $\ell$ number field where $\ell$ is an odd prime
that is not ramified in $L/\QQ$.
Let $K_0 = \QQ(i)$ and let $K_\infty/K_0$ be the anti-cyclotomic $\ZZ_2$-extension.
Then let $L_n = K_n L$ so that $L_\infty/L_0$ is also a $\ZZ_2$-extension.

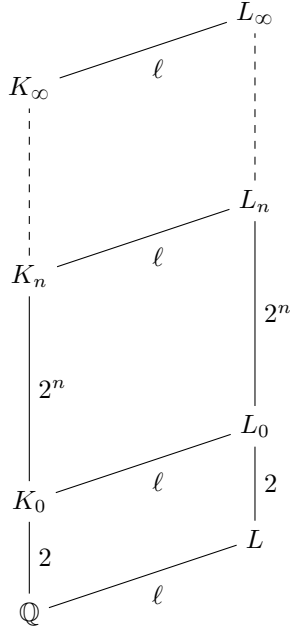
\begin{center}
\begin{tikzpicture}[node distance = 1.55cm, auto]
    \node (Q) at (0,0) {$\QQ$};
    \node (K0) at (0,1.5) {$K_0$};
    \node (Kn) at (0,4.5) {$K_n$};
    \node (Kinfty) at (0,7) {$K_\infty$};
    \node (L) at (3,1) {$L$};
    \node (L0) at (3,2.5) {$L_0$};
    \node (Ln) at (3,5.5) {$L_n$};
    \node (Linfty) at (3,8) {$L_\infty$};
    \draw[-] (Q) to node [swap] {$2$} (K0);
    \draw[-] (K0) to node [swap] {$2^{n}$} (Kn);
    \draw[dashed] (Kn) to node [swap] {} (Kinfty);
    \draw[-] (Q) to node [swap] {$\ell$} (L);
    \draw[-] (L) to node [swap] {$2$} (L0);
    \draw[-] (L0) to node [swap] {$2^{n}$} (Ln);
    \draw[dashed] (Ln) to node [swap] {} (Linfty);
    \draw[-] (K0) to node [swap] {$\ell$} (L0);
    \draw[-] (Kn) to node [swap] {$\ell$} (Ln);
    \draw[-] (Kinfty) to node [swap] {$\ell$} (Linfty);
\end{tikzpicture}
\end{center}

We wish to study the $\ell$-class groups of $L_n$. One of our primary tools is
Chevalley's formula, given below.
\begin{theorem}[Chevalley's Formula, \cite{chevalleypaper}]\label{chevalley}
Let $L/K$ be a cyclic extension of number fields; $\Delta = \Gal(L/K)$; $n=[L:K]$; $C_L$ be the ideal class group of $L$; $h(K)$ be the class number of $K$; $e(L/K) = \prod_P e_P$ be the product over all primes $P$ of $K$, including archimedean ones, where $e_P$ is the ramification index of $P$ in $L/K$; and $E(L/K) = [E_K:E_K \cap N_{L/K}(L^\times)]$, where $E_K$ is the group of units of $K$. Then
\begin{align*}
 \big | C_L^\Delta \big | = \frac{h(K)\cdot e(L/K)}{n \cdot E(L/K)}.
\end{align*}
\end{theorem}

Let $A_{n}$ be the $\ell$-class group of $L_n$ and let $h_\ell(K_n)$
be the $\ell$-part of the class number of $K_n$.
Assume there are $t$ rational primes that ramify in $L/\QQ$ and that they
are all inert in $K_0/\QQ$. Therefore there are $2^n t$ primes in $K_n$
that ramify in $L_n/K_n$ (Lemma \ref{splitcompletely}).
So $$e(L_n/K_n) = \ell^{2^n t}.$$
Letting $h_\ell(K_n)$ be the $\ell$-class number of $K_n$,
 $$\big | A_{n}^\Delta \big | = \frac{h_\ell(K_n)\cdot \ell^{2^n t -1}}{E(L_n/K_n)}.$$
As discussed in \S \ref{sec:Kn}, there are $2^n-1$ fundamental units in $K_n$ and no
non-trivial $\ell^{th}$ roots of unity.
Therefore $$1 \leq E(L_n/K_n) \leq \ell^{2^n-1}.$$
So we know immediately that
$$\big | A_{n} \big | \geq \big | A_{n}^\Delta \big | \geq \ell^{2^n t -2^n}.$$
For $t > 1$, we therefore can see that the order of the $\ell$-class group grows with $n$.

This is the motivation for this paper. We've shown that there exist $\ZZ_2$-extensions
where the $\ell$-part of the class number is unbounded using a trivial estimate
for the unit norm index.
But in order to understand the causes of such examples, we need to
study the unit norm index more closely.

\section{Norms of Units}\label{sec:norms}

We now address the behavior of the unit norm index $E(L_n/K_n)$.

Let $L_n/K_n$ be a cyclic prime extension of degree $\ell$.
We have the following proposition which tells us that we only need to worry about
the `new units' at each stage.

\begin{lemma}\label{normprop}
For $n\geq 1, m\geq 1$, let $u$ be a unit in $K_n$ regarded as a unit in $K_{n+m}$.
Then $u$ is the norm of an element in $L_{n+m}$ if and only if it is the norm of
an element in $L_n$.
\end{lemma}
\begin{proof}
Suppose $u \in K_{n+m}$ is the norm of an element
$a \in L_{n+m}$: $$N_{L_{n+m}/K_{n+m}}(a) = u.$$
Then $$ N_{L_{n+m}/K_n} (a) = N_{K_{n+m}/K_{n}} N_{L_{n+m}/K_{n+m}} (a) = u^{2^m} .$$
Since $2^m$ is prime to $\ell$ and we have a degree $\ell$ extension, this implies $u$ is a norm.
Therefore $u$ is the norm of an element in $L_n$.
The converse is trivial.
\end{proof}

Recall that by the Hasse Norm Theorem, a unit $u \in K_n$ is the norm of
an element in $L_n$ if and only if it is a local norm at all primes that ramify
in $L_n/K_n$.
Therefore we can determine if the units of $K_n$ are norms from $L_n$
by constructing a matrix of norm residue symbols where the $(i,j)^{th}$ entry of the matrix
is the symbol $\left( \frac{u_j}{\mfp_i} \right)$.
We will use this matrix to prove the following theorem, our main result.
\begin{theorem}\label{allornothing} 
Let $n \geq 2$ and $\ell \equiv 3 \text{ or } 5 \mmod 8$.
Let $K_n$ be the $n^{th}$ layer in the anti-cyclotomic $\ZZ_2$-extension of $K_0 = \QQ(i)$.
Let $L_n/K_n$ be a cyclic degree $\ell$ extension that is a lift of a cyclic
degree $\ell$ extension $L/\QQ$.
Assume that the ramified primes in $L/\QQ$ are $3 \mmod 4$ and assume $\ell$ is unramified in $L/\QQ$.
Let $U_n^{rel}$ be the units in $K_n$ that
have norm 1 in $K_{n-1}$ modulo $\ell^{th}$ powers.
Then either all of the relative units in $K_n$ are norms of elements in $L_n$
or none of them are (or more precisely, the only
elements that are norms are $\ell^{th}$ powers).
\end{theorem}
In order to prove this theorem, we need some results concerning skew circulant matrices.
We present these results in \S \ref{sec:skew} followed by the proof in \S \ref{sec:proof}.

\subsection{Skew Circulant Matrices}\label{sec:skew}
The matrices of norm residue symbols will turn out to have a special form:
they will be skew circulant matrices.

\begin{definition}\label{defskew}
A \emph{skew circulant} matrix is an $n \times n$ matrix of the form
$$\begin{pmatrix}
a_0 & a_1 & \ldots & a_{n-2} & a_{n-1} \\
-a_{n-1} & a_0 & \ldots & a_{n-3} & a_{n-2} \\
-a_{n-2} & -a_{n-1} & \ldots & a_{n-4} & a_{n-3} \\
\vdots & \vdots & \ddots & \vdots & \vdots \\
-a_2 & -a_3 & \ldots & a_0 & a_1\\
-a_1 & -a_2 & \ldots & -a_{n-1} & a_0\\
\end{pmatrix}$$
\end{definition}

The associated polynomial to a skew circulant matrix is
$$g(x) = a_0 + a_1 x + \ldots + a_{n-1} x^{n-1}.$$
Then for $0 \leq i \leq n-1$, each row of the matrix is given by $$x^i g(x) \mmod (x^n+1).$$

\begin{proposition}\label{negaeigen}
Let $\zeta$ be a solution to $x^n + 1 = 0$ ($\zeta$ is a $(2n)^{th}$ root of unity, not necessarily primitive).
Each $\zeta$ gives an eigenvector of a skew circulant matrix
$$(1,\zeta, \zeta^{2}, \ldots, \zeta^{n-1})$$
and the eigenvalues are $g(\zeta)$ where $g$ is the associated polynomial.
\end{proposition}
\begin{proof}
Consider an $n \times n$ skew circulant matrix.
Then
\begin{align*}
&\begin{pmatrix}
a_0 & a_1 & \ldots & a_{n-2} & a_{n-1} \\
-a_{n-1} & a_0 & \ldots & a_{n-3} & a_{n-2} \\
\vdots & \vdots & \ddots & \vdots & \vdots \\
-a_1 & -a_2 & \ldots & -a_{n-1} & a_0\\
\end{pmatrix}
\begin{pmatrix}1 \\ \zeta \\ \vdots \\  \zeta^{n-1}
\end{pmatrix} \\
&=
\begin{pmatrix}
a_0 + a_1 \zeta + \ldots + a_{n-2} \zeta^{n-2} + a_{n-1} \zeta^{n-1} \\
-a_{n-1} + a_0 \zeta + \ldots + a_{n-3} \zeta^{n-2} + a_{n-2} \zeta^{n-1} \\
\vdots \\
-a_1 - a_2 \zeta - \ldots - a_{n-1} \zeta^{n-2} + a_0 \zeta^{n-1} \\
\end{pmatrix} \\
&=
(a_0+a_1 \zeta + \ldots + a_{n-2}\zeta^{n-2} + a_{n-1} \zeta^{n-1})
\begin{pmatrix}1 \\ \zeta \\ \vdots \\  \zeta^{n-1}
\end{pmatrix}
\end{align*}
\end{proof}

\subsection{Proof of Theorem \ref{allornothing}}\label{sec:proof}
We're now ready to prove the theorem.

There are $2^{n-1}$ generators of the relative units modulo $\ell^{th}$ powers in $K_n$.
By Corollary \ref{relative_unit_fixed}, at least one
relative unit is fixed by $\tau$. By Proposition \ref{gens_relative_units},
we may choose $u$ to be a generator that is fixed by $\tau$ and then
$$\{u, \hspace{.1in} \sigma u, \hspace{.1in} \sigma^2 u, \ldots, \hspace{.1in} \sigma ^{2^{n-1}-1} u\}$$ is a set of generators for $U_n^{rel}$.
These are the units in $K_n$ that have norm 1 in $K_{n-1}$ modulo $\ell^{th}$ powers. By
Proposition \ref{unitsdirectsum} and Lemma
\ref{normprop}, these are the only units we need to consider, as we can assume we already know if
units from fields lower in the tower are norms.

A unit $u \in K_n$ is the norm of an element in $L_n$ if and only if its norm
residue symbols for each prime that ramifies in $L_n/K_n$ are trivial.
Let
$u_j = \sigma^j u$.
Observe that this implies
$$\tau u_j = \tau \sigma^j u = \sigma^{-j} \tau u = \sigma^{2^n-j} u = (\sigma^{2^{n-1}-j} u)^{-1} = u_{2^{n-1}-j}^{-1}$$
since $\sigma^{2^{n-1}} u = u^{-1}$ (Proposition \ref{gens_relative_units}).

We can consider each set of conjugate primes separately, so we'll first look
only at primes that all lie over the same rational prime $p$.
Write these primes as
$$\mfp_0 = \mfp, \hspace{.1in} \mfp_1 = \sigma \mfp, \hspace{.1in} \ldots, \hspace{.1in} \mfp_{2^n-1} = \sigma^{2^n-1} \mfp.$$
We know there are $2^n$ primes in $K_n$ lying over $p$,
using Lemma \ref{splitcompletely} and our assumption that $p \equiv 3 \mmod 4$.

We will make use of the following lemma:
\begin{lemma}\label{lemmafactor}
Let $n>2$.

If $\ell \equiv 3 \mmod 8$,
$$x^{2^{n-1}}+1 \equiv (x^{2^{n-2}}+\sqrt{-2}x^{2^{n-3}}-1)(x^{2^{n-2}}-\sqrt{-2}x^{2^{n-3}}-1) \mmod \ell.$$

If $\ell \equiv 5 \mmod 8$,
$$x^{2^{n-1}}+1 \equiv (x^{2^{n-2}}+\sqrt{-1})(x^{2^{n-2}}-\sqrt{-1}) \mmod \ell.$$
\end{lemma}
\begin{proof}
Since $\sqrt{-2} \in \FF_\ell$ for $\ell \equiv 3 \mmod 8$ and
$\sqrt{-1} \in \FF_\ell$ for $\ell \equiv 5 \mmod 8$,
the proof is a straightforward computation.
\end{proof}

We'll address four cases independently: we treat  $\ell \equiv 3 \mmod 8$ and $\ell \equiv 5 \mmod 8$
separately, and we treat $p \equiv 7 \mmod 8$ and $p \equiv 3 \mmod 8$ separately.
\begin{description}[style=unboxed,leftmargin=0cm]
\item[$\bullet \hspace{1em} p \equiv 7 \mmod 8$, $\ell \equiv 3 \mmod 8$]
If $p \equiv 7 \mmod 8$, then by Proposition \ref{7fixprime19}
at least one of the primes above $p$ is fixed by $\tau$.
Let $\mfp$
be the fixed prime: $\tau \mfp = \mfp$. Then
$$\tau \sigma^j \mfp = \sigma^{-j} \tau \mfp = \sigma^{-j} \mfp.$$

Construct a matrix $(a_{ij})$ using the norm residue symbols:
\begin{align*}
\left(\frac{u_j}{\mfp_i}\right) \equiv u_j^{(N\mfp_i-1)/\ell} \equiv \omega^{a_{ij}} \mmod \mfp_i
\end{align*}
where $\omega$ is a fixed primitive $\ell^{th}$ root of unity. We extend $\tau$
and $\sigma$ such that $\omega$ is fixed by both.
Since each $\mfp_i$ is inert in $K_0/\QQ$ and totally split in $K_n/K_0$, $N\mfp_i = p^2$.
Since we're considering the powers of $\ell^{th}$ roots of unity,
the matrix entries $a_{ij}$ are in $\FF_\ell$.

If the matrix is trivial, then all of the norm residue symbols are trivial, and therefore
all of the units are norms modulo $\ell^{th}$ powers. If the matrix has full rank,
then none of the units are norms modulo $\ell^{th}$ powers.

First consider $\mfp_0 = \mfp$.
Write
$$\left( \frac{u_j}{\mfp} \right) = \omega^{a_j}.$$
For $0 \leq j \leq 2^{n-1}-1$,
apply $\tau$, using $\tau u_j = u_{2^{n-1}-j}^{-1}$, $\tau \omega = \omega$
and $\tau \mfp = \mfp$:
\begin{align*}
\omega^{a_j} =
\left( \frac{u_{j}}{\mfp} \right) =
\tau \left( \frac{u_{j}}{\mfp} \right) =
\left( \frac{u_{2^{n-1}-j}^{-1}}{\mfp} \right) =
\omega^{-a_{2^{n-1}-j}}
\implies a_j = -a_{2^{n-1}-j}.\end{align*}
We want to pay special attention to
what happens when $j = 2^{n-1}-j$, or $j = 2^{n-2}$.
Then  $$a_{2^{n-2}} = -a_{2^{n-2}} \implies a_{2^{n-2}} = 0.$$

We can use the above relations to write the first row of the matrix $(a_{ij})$ in terms
of $a_j$, $0 \leq j < 2^{n-2}$:
$$\begin{pmatrix}
a_0 & a_1 & \ldots & a_{2^{n-2}-1} & 0  & -a_{2^{n-2}-1} & \ldots & -a_1
\end{pmatrix}.$$

To get the next row, we can apply $\sigma$.
Then
$$\omega^{a_j} = \left( \frac{u_{j}}{\mfp} \right) = \sigma \left( \frac{u_{j}}{\mfp} \right) = \left( \frac{u_{j+1}}{\sigma \mfp} \right)$$
for $0 \leq j < 2^{n-1}-1$ and
$$\omega^{a_1} = \left( \frac{u^{-1}_{2^{n-1}-1}}{\mfp} \right) = \sigma \left( \frac{u^{-1}_{2^{n-1}-1}}{\mfp} \right) = \left( \frac{u_0}{\sigma \mfp} \right).$$

If we continue in this manner, we can construct
a $2^n \times 2^{n-1}$ matrix, where the top half and bottom half are both
skew circulant:
\begin{equation}\label{fixedprimematrix}
\begin{pmatrix}
a_0 & a_1 & \ldots & a_{2^{n-2}-1} & 0 & -a_{2^{n-2}-1} & \ldots & -a_1 \\
a_1 & a_0 & \ldots & a_{2^{n-2}-2} & a_{2^{n-2}-1} & 0 & \ldots & -a_2 \\
a_2 & a_1 & \ldots & a_{2^{n-2}-3} & a_{2^{n-2}-2} & a_{2^{n-2}-1} & \ldots & -a_3\\
\vdots & \vdots & \ddots & \vdots & \vdots & \vdots & \ddots & \vdots \\
-a_1 & -a_2 & \ldots & 0 & a_{2^{n-2}-1} & a_{2^{n-2}-2} & \ldots & a_0 \\
-a_0 & -a_1 & \ldots & -a_{2^{n-2}-1} & 0 & a_{2^{n-2}-1} & \ldots & a_1 \\
\vdots & \vdots & \ddots & \vdots & \vdots & \vdots & \ddots & \vdots \\
a_2 & a_3 & \ldots & -a_{2^{n-2}-1}  & -a_{2^{n-2}-2} & -a_{2^{n-2}-3} & \ldots & -a_1 \\
a_1 & a_2 & \ldots & 0 & -a_{2^{n-2}-1} & -a_{2^{n-2}-2} & \ldots & -a_0 \\
\end{pmatrix}
\end{equation}

If $n = 2$, we have the matrix
$$\begin{pmatrix}
a_0 & 0 \\
0 & a_0 \\
-a_0 & 0 \\
0 & -a_0
\end{pmatrix}$$
which has rank 0 if $a_0 = 0$ and rank $2$ if $a_0 \neq 0$.
This means that either both units are local norms at all four primes or
neither of them are.

Now we consider $n>2$.
By Proposition \ref{negaeigen}, the top half of the matrix has eigenvectors
$$(1, \zeta, \ldots, \zeta^{2^{n-1}-1})$$
where $\zeta$ is one of the $2^{n-1}$ distinct
solutions to $x^{2^{n-1}}+1=0$.
Note that $$[\FF_\ell(\zeta_{2^{n}}):\FF_\ell] = 2^{n-2}$$
for $\ell \equiv \pm 3 \mmod 8$.
The corresponding eigenvalue is
\begin{align*}
& a_0 + a_1 \zeta + \ldots + a_{2^{n-2}-1} \zeta^{2^{n-2}-1} - a_{2^{n-2}-1} \zeta^{2^{n-2}+1} - \ldots - a_1 \zeta^{2^{n-1}-1} \\
&= \sum_{j=0}^{2^{n-2}-1} a_j \zeta^j - \sum_{j=1}^{2^{n-2}-1} a_{2^{n-2}-j} \zeta^{2^{n-2}+j}.
\end{align*}
We want to show that the eigenvalues are all non-zero unless the matrix
identically zero.

By Lemma \ref{lemmafactor}, $\zeta^{2^{n-2}} \equiv \pm \sqrt{-2} \zeta^{2^{n-3}} + 1$.
Without loss of generality, we can assume $$\zeta^{2^{n-2}} \equiv \sqrt{-2} \zeta^{2^{n-3}} + 1$$
by choosing the sign of $\sqrt{-2}$.

Then
$$\zeta^{2^{n-2}+2^{n-3}} \equiv \sqrt{-2} \zeta^{2^{n-2}} + \zeta^{2^{n-3}} \equiv
  -\zeta^{2^{n-3}} + \sqrt{-2}.$$

We can use this to reduce the eigenvalue to lower terms.
\begin{align*}
& \sum_{j=0}^{2^{n-2}-1} a_j \zeta^j - \sum_{j=1}^{2^{n-2}-1} a_{2^{n-2}-j} \zeta^{2^{n-2}+j} \\
& = a_0 + a_{2^{n-3}} \zeta^{2^{n-3}} + \sum_{j=1}^{2^{n-3}-1} (a_j \zeta^j + a_{2^{n-3}+j} \zeta^{2^{n-3}+j})
\\
& \hspace{.1in} -a_{2^{n-3}} \zeta^{2^{n-2}+2^{n-3}} - \sum_{j=1}^{2^{n-3}-1} (a_{2^{n-2}-j} \zeta^{2^{n-2}+j} + a_{2^{n-3}-j} \zeta^{2^{n-2}+2^{n-3}+j}) \\
& = a_0  -\sqrt{-2} a_{2^{n-3}} + 2 a_{2^{n-3}} \zeta^{2^{n-3}}
+ \sum_{j=1}^{2^{n-3}-1} \big( (a_j - a_{2^{n-2}-j}  - \sqrt{-2} a_{2^{n-3}-j}) \zeta^j \big) \\
& \hspace{.2in} + \sum_{j=1}^{2^{n-3}-1} \big(  (a_{2^{n-3}+j} - \sqrt{-2} a_{2^{n-2}-j}  +  a_{2^{n-3}-j}) \zeta^{2^{n-3}+j} \big) \\
\end{align*}

The eigenvalue is a polynomial in $\zeta$ of degree $2^{n-2}-1$.
Therefore we can write the eigenvalue as a $2^{n-2} \times 2^{n-2}$ matrix, which we'll call $B$,
multiplied by the column vector $(a_0, a_1, \ldots, a_{2^{n-2}-1})$.
Row $j$ corresponds to the coefficient of $\zeta^j$.
We write $I$ to represent the identity matrix and $J$ to represent the anti-identity matrix.
The anti-identity matrix is all zero except for ones on the diagonal that goes from the
upper right corner to the lower left corner.

Then
$$B =
\left( \begin{array}{c|c|c|c} 1& & - \sqrt{-2} & \\
\hline & I - \sqrt{-2} J & & -J \\
\hline 0 & & 2 & \\
\hline & J & & I - \sqrt{-2} J
\end{array} \right)$$
where the $I$ and $J$ matrices are both of size $(2^{n-3}-1) \times (2^{n-3}-1)$.
The eigenvalue is zero if and only if the product of $B$ and
the column vector $(a_0, a_1, \ldots, a_{2^{n-2}-1})$ is zero.

\begin{lemma}\label{rankB}
The matrix
$$
\left(\begin{array}{c|c} I - \sqrt{-2} J & -J \\ \hline J & I - \sqrt{-2} J \end{array} \right)$$
has full rank.
\end{lemma}
\begin{proof}
We have the following chain of
row operations:
\begin{align*}
&\left(\begin{array}{c|c} I - \sqrt{-2} J & -J \\ \hline J & I - \sqrt{-2} J \end{array} \right)\\
& \rightarrow
\left(\begin{array}{c|c} I - \sqrt{-2} J & -J \\ \hline \sqrt{-2} J & \sqrt{-2} I + 2 J \end{array} \right)
\hspace{.3in} \text{(multiply 2nd row by } \sqrt{-2} \text{)}
\\
& \rightarrow
\left(\begin{array}{c|c} I  & \sqrt{-2} I + J \\ \hline \sqrt{-2} J & \sqrt{-2} I + 2 J \end{array} \right)
\hspace{.3in} \text{(add 2nd row to 1st)}
\\
& \rightarrow
\left(\begin{array}{c|c} I  & \sqrt{-2} I + J \\ \hline I & J - \sqrt{-2} I \end{array} \right)
\hspace{.3in} \text{(multiply 2nd row by } \frac{1}{\sqrt{-2}} J \text{)}
\\
& \rightarrow
\left(\begin{array}{c|c} I & I + J \\ \hline 0 & - 2\sqrt{-2} I \end{array} \right)
\hspace{.3in} \text{(2nd row - 1st row } \text{)}
\end{align*}
This matrix has non-zero determinant, and therefore $B$ has non-zero determinant.
\end{proof}
Therefore $B$ has full rank, and so there are no non-trivial solutions to
$$B \cdot (a_0, \ldots, a_{2^{n-2}-1})^T = (0,\ldots,0).$$
Therefore the eigenvalue equals zero only if $a_j = 0$ for all $j$.
So either the matrix of norm residue symbols (Equation \ref{fixedprimematrix}) is identically
zero or it has full rank.

\vspace{.1in}
\item[$\bullet \hspace{1em} p \equiv 3 \mmod 8$, $\ell \equiv 3 \mmod 8$]
Since $p \equiv 3 \mmod 8$, no primes lying over $p$ are fixed by $\tau$.
Without loss of generality, we can choose $\mfp$ so that $\tau \mfp = \sigma \mfp$. Then
$$\tau \sigma^j \mfp = \sigma^{-j} \tau \mfp = \sigma^{-j+1} \mfp.$$

Again construct a matrix $a_{ij}$ over $\FF_\ell$ using the norm residue symbols:
$$\left(\frac{u_j}{\mfp_i}\right) \equiv \omega^{a_{ij}} \mmod \mfp_i$$
where $\omega$ is a fixed primitive $\ell^{th}$ root of unity.

We'll again start by considering $\mfp_0 = \mfp$.
Write
$$\left( \frac{u_j}{\mfp} \right) = \omega^{a_j}.$$

Then using $\tau u_j = u_{2^{n-1}-j}^{-1}$ and
the assumption that $\tau \mfp = \sigma \mfp$, for $0 \leq j < 2^{n-2}$,
$$\omega^{a_j} = \left ( \frac{u_j}{\mfp} \right) = \tau \left( \frac{u_j}{\mfp} \right)
=\left( \frac{u_{2^{n-1}-j}}{\sigma \mfp} \right)^{-1}.$$
Apply $\sigma^{2^n-1}$:
$$\omega^{a_j} = \sigma^{2^n-1} (\omega^{a_j} ) =
\left( \frac{u^{-1}_{2^{n-1}-j-1}}{\mfp} \right)
= \omega^{-a_{2^{n-1}-j-1}}.$$

Therefore the first row of the matrix is
$$\begin{pmatrix}
a_0 & a_1 & \ldots & a_{2^{n-2}-1} & -a_{2^{n-2}-1} & \ldots & -a_0
\end{pmatrix}.$$

We can again apply $\sigma$ to get the full matrix:
\begin{equation}\label{nofixedprimematrix}
\begin{pmatrix}
a_0 & a_1 & \ldots & a_{2^{n-2}-1} & -a_{2^{n-2}-1} & \ldots & -a_0 \\
a_0 & a_0 & \ldots & a_{2^{n-2}-2} & a_{2^{n-2}-1} & \ldots & -a_1 \\
a_1 & a_0 & \ldots & a_{2^{n-2}-3} & a_{2^{n-2}-2} & \ldots & -a_2 \\
\vdots & \vdots & \ddots & \vdots & \vdots & \ddots & \vdots \\
a_2 & a_3 & \ldots & -a_{2^{n-2}-2} & -a_{2^{n-2}-3} & \ldots & -a_1 \\
a_1 & a_2 & \ldots & -a_{2^{n-2}-1} & -a_{2^{n-2}-2} & \ldots & -a_0 \\
\end{pmatrix}
\end{equation}

If $n = 2$, we have the matrix
$$\begin{pmatrix}
a_0 & -a_0 \\
a_0 & a_0 \\
-a_0 & a_0 \\
-a_0 & -a_0
\end{pmatrix}$$
which has rank 0 if $a_0 = 0$ and rank $2$ if $a_0 \neq 0$.

For $n>2$, we again apply Proposition \ref{negaeigen}.
The top half of the matrix has eigenvectors
$$(1, \zeta, \ldots, \zeta^{2^{n-1}-1})$$
where $\zeta$ is one of the $2^{n-1}$ distinct
solutions to $x^{2^{n-1}}+1=0$. The corresponding eigenvalue is
\begin{align*}
& a_0 + a_1 \zeta + \ldots + a_{2^{n-2}-1} \zeta^{2^{n-2}-1} - a_{2^{n-2}-1} \zeta^{2^{n-2}} - \ldots - a_0 \zeta^{2^{n-1}-1} \\
&= \sum_{j=0}^{2^{n-2}-1} a_j \zeta^j - \sum_{j=0}^{2^{n-2}-1} a_{2^{n-2}-1-j} \zeta^{2^{n-2}+j}.
\end{align*}

For $n>2$, we again have
\begin{align*}
x^{2^{n-1}}+1 &\equiv (x^{2^{n-2}}+ \sqrt{-2} x^{2^{n-3}}-1)(x^{2^{n-2}}- \sqrt{-2} x^{2^{n-3}}-1), \\
\zeta^{2^{n-2}} &\equiv \sqrt{-2} \zeta^{2^{n-3}} + 1, \\
\zeta^{2^{n-2}+2^{n-3}} &\equiv -\zeta^{2^{n-3}} + \sqrt{-2}.
\end{align*}

We can use this to reduce the eigenvalue to lower terms.
\begin{align*}
& \sum_{j=0}^{2^{n-2}-1} a_j \zeta^j - \sum_{j=0}^{2^{n-2}-1} a_{2^{n-2}-1-j} \zeta^{2^{n-2}+j} \\
&= \sum_{j=0}^{2^{n-3}-1} \left( a_j \zeta^j + a_{2^{n-3}+j} \zeta^{2^{n-3}+j}  -  a_{2^{n-2}-1-j} \zeta^{2^{n-2}+j} -  a_{2^{n-3}-1-j} \zeta^{2^{n-2}+2^{n-3}+j} \right) \\
&\equiv \sum_{j=0}^{2^{n-3}-1} \big( (a_j- a_{2^{n-2}-1-j}- \sqrt{-2} a_{2^{n-3}-1-j}) \zeta^j \\
&\hspace{.1in} + (a_{2^{n-3}+j}- \sqrt{-2} a_{2^{n-2}-1-j} + a_{2^{n-3}-1-j} ) \zeta^{2^{n-3}+j} \big) \\
\end{align*}

Writing the matrix $B$ corresponding to the eigenvector
as in the previous case, we obtain
$$B =
\left( \begin{array}{c|c}
I - \sqrt{-2} J & -J \\ \hline J & I - \sqrt{-2} J
\end{array} \right)$$
where the $I$ and $J$ matrices are both of size $(2^{n-3}) \times (2^{n-3})$.
So by Lemma \ref{rankB}, the matrix has full rank and so there are no non-trivial
solutions for the $\{a_j\}$ that yield a zero eigenvalue.

\vspace{.1in}
\item[$\bullet \hspace{1em} p \equiv 7 \mmod 8$, $\ell \equiv 5 \mmod 8$]

In this case, we again have a fixed prime: $\tau \mfp = \mfp$. Then
$$\tau \sigma^j \mfp = \sigma^{-j} \tau \mfp = \sigma^{-j} \mfp.$$
We then construct the matrix $(a_{ij})$ using the norm residue symbols
to get the same skew circulant matrix as in Equation \ref{fixedprimematrix}.

The eigenvalues of the matrix are of the form
\begin{align*}
\sum_{j=0}^{2^{n-2}-1} a_j \zeta^j - \sum_{j=1}^{2^{n-2}-1} a_{2^{n-2}-j} \zeta^{2^{n-2}+j}.
\end{align*}
for $\zeta$ a solution to $x^{2^{n-1}}+1=0$.
We want to show that the eigenvalues are all non-zero unless the matrix is the
zero matrix.

The case $n=2$ is again trivial.
For $n>2$, we use Lemma \ref{lemmafactor},
which tells us that $\zeta^{2^{n-2}} \equiv \pm \sqrt{-1}$.
Again, we assume without loss of generality that $\zeta^{2^{n-2}} \equiv \sqrt{-1}$.

We can use this to reduce the eigenvalue to lower terms.
\begin{align*}
& \sum_{j=0}^{2^{n-2}-1} a_j \zeta^j - \sum_{j=1}^{2^{n-2}-1} a_{2^{n-2}-j} \zeta^{2^{n-2}+j} \\
\equiv & \sum_{j=0}^{2^{n-2}-1} a_j \zeta^j - \sqrt{-1} \sum_{j=1}^{2^{n-2}-1} a_{2^{n-2}-j} \zeta^{j} \\
\equiv & \hspace{.1in} a_0 + \sum_{j=1}^{2^{n-2}-1} (a_j - \sqrt{-1} a_{2^{n-2}-j}) \zeta^j
\end{align*}

We write the eigenvalue as a $2^{n-2} \times 2^{n-2}$ matrix $B$ multiplied by the column vector $(a_0, a_1, \ldots, a_{2^{n-2}-1})$ just as in the previous cases.

Then
$$B =
\left( \begin{array}{c|c|c|c} 1& & 0 &  \\
\hline & I & &  - \sqrt{-1} J \\
\hline 0 &  & 1 - \sqrt{-1} & \\
\hline &  - \sqrt{-1} J & & I  \\
\end{array} \right)
$$
where the $I$ and $J$ matrices are both of dimension $(2^{n-3}-1) \times (2^{n-3}-1)$.

\begin{lemma}\label{rankB2}
The matrix
$$\left(\begin{array}{c|c} I  & -\sqrt{-1} J \\ \hline -\sqrt{-1} J  & I  \end{array} \right)$$
has full rank.
\end{lemma}
\begin{proof}
We have the following chain of
row operations:

\begin{align*}
& \left(\begin{array}{c|c} I  & -\sqrt{-1} J \\ \hline -\sqrt{-1} J  & I  \end{array} \right) \\
& \rightarrow
\left(\begin{array}{c|c} I & -\sqrt{-1} J \\ \hline -I  &  - \sqrt{-1} J  \end{array} \right)
\hspace{.3in} \text{(multiply 2nd row by } - \sqrt{-1} J \text{)}
\\
& \rightarrow
\left(\begin{array}{c|c} I & -\sqrt{-1} J \\ \hline 0  &  - 2 \sqrt{-1} J  \end{array} \right)
\hspace{.3in} \text{(add 1st row to 2nd)}
\\
& \rightarrow
\left(\begin{array}{c|c} I & -\sqrt{-1} J \\ \hline 0  & 2I  \end{array} \right)
\hspace{.3in} \text{(multiply 2nd row by } \sqrt{-1} J  \text{)}
\\
\end{align*}
This matrix has non-zero determinant.
\end{proof}
Therefore $B$ has full rank and so the only way for the
eigenvalue to equal zero is if $a_j = 0$ for all $j$.

\vspace{.1in}
\item[$\bullet \hspace{1em} p \equiv 3 \mmod 8$, $\ell \equiv 5 \mmod 8$]
Here, the matrix $(a_{ij})$ derived from the norm residue symbols is of the
same form as in Equation \ref{nofixedprimematrix}. Its eigenvalues are of the form
\begin{align*}
\sum_{j=0}^{2^{n-2}-1} a_j \zeta^j - \sum_{j=0}^{2^{n-2}-1} a_{2^{n-2}-1-j} \zeta^{2^{n-2}+j}.
\end{align*}
By Lemma \ref{lemmafactor}, $\zeta^{2^{n-2}} \equiv  \sqrt{-1}$.
Therefore the eigenvalue can be reduced to
\begin{align*}
\sum_{j=0}^{2^{n-2}-1} (a_j - \sqrt{-1} a_{2^{n-2}-1-j}) \zeta^{j}.
\end{align*}

This eigenvalue, when embedded into a matrix $B$, is of the form
$$\left(\begin{array}{c|c} I  & -\sqrt{-1} J \\ \hline -\sqrt{-1} J  & I  \end{array} \right)$$
where $I$ and $J$ have dimension $2^{n-3} \times 2^{n-3}$.
By Lemma \ref{rankB2}, we know this matrix has non-zero determinant.
\end{description}

In all four cases, the eigenvalues are all non-zero unless the matrix
is identically zero. Therefore
the matrix of norm residue symbols is either identically zero or it has full rank.

If primes over more than one rational prime ramify, we can construct matrices
for each set of conjugate primes separately. Then each matrix is either full
rank or rank zero. Therefore at a given prime,
the units are either all local norms or none
of them are. If the units are all local norms modulo all of the primes,
the units are global norms. Otherwise, none of the units are global norms.
\vspace{-\baselineskip}
\hspace{6in} \qed
\vspace{\baselineskip}

The theorem is false for $\ell \equiv 1 \text{ or } 7 \mmod 8$.
For example, consider the degree 7 field with defining polynomial
$$x^7 - x^6 - 54x^5 + 31x^4 + 558x^3 + 32x^2 - 1713 x - 1121$$
which has discriminant $127^6$. The relative unit group of $K_3$ has four
generators. In this case, the matrix of norm residue symbols has rank 2.

\section{The $\ell$-Class Groups of $L_n$}\label{sec:app}

We now develop a model for how often we should expect the units in
$K_n$ to be the norms of elements in $L_n$, where $K_n$ is the $n^{th}$
layer of the anti-cyclotomic $\ZZ_2$-extension of $\QQ(i)$ and $L_n/K_n$
is the lift of a cyclic degree $\ell$ extension $L/\QQ$ of prime degree $\ell \neq 2$.
We use the same assumptions
as in the statement of Theorem \ref{allornothing}.
We also assume that $\ell \nmid h(K_n)$.
We expect that the norm residue symbols are
equidistributed, and that we can therefore use
the results of Theorem \ref{allornothing} to
derive the probabilities that the units are norms.

\subsection{The $n=0$ Case}\label{n0case}

The $K_0$ case is trivial, since there are no fundamental units in $K_0$.
Therefore by Theorem \ref{chevalley},
$\rk A_0^\Delta = 0$. We can then apply the following proposition.

\begin{proposition}[Gras \cite{gras1972}, Proposition 4.1]\label{grasprop}
Let $A$ be the $\ell$-class group of a cyclic degree $\ell$ extension $L/K$
with Galois group $\langle \sigma \rangle$ and  $\ell \nmid h(K)$.
Let $$A^j = \{a \in A | a^{(\sigma-1)^j} = 1 \}$$ for $j \geq 0.$
Then $A^j \subseteq A^{j+1}$ and $A^j = A^{j+1}$ if and only if $A^j=A$.
\end{proposition}
Note that $A_0^\Delta = A_0^1$
Therefore if $\rk A_0^\Delta = 0$ then we must have $\rk A_0 = 0$.

\subsection{The $n=1$ Case}\label{n1case}
There is one fundamental unit in $K_1 = \QQ(\zeta_8)$. We choose $\sqrt{2}-1$
to be the generator of the unit group.
To determine if it is the norm of an element in
$L_1$, we compute norm residue symbols for the ramified primes.
It happens that the $n=1$ case is special, as seen in the following result.
\begin{proposition}\label{trivialsymbols}
Let $K_1 = \QQ(\zeta_8)$ and let $L_1/K_1$ be the lift of a cyclic degree $\ell$
extension $L/\QQ$, where $\ell$ is an odd prime and primes lying over $\ell$ do not
ramify in $L_1/K_1$.
Let $p \equiv 3 \mmod 8$ be a rational prime that ramifies in
$L_1/K_1$. The norm residue symbols of the units of $K_1$ are trivial for
the primes above $p$.
\end{proposition}
\begin{proof}
To prove the theorem, we let $L'_1/K_1$ be the lift of a cyclic degree $\ell$
extension $L'/\QQ$ in which only primes above
$p$ ramify. Since $L'_1$ is the lift of an abelian degree $\ell$ number field,
$L'_1 \subseteq \QQ(\zeta_{8p})$. Since $1-\zeta_{8p}$ is a unit in $\QQ(\zeta_{8p})$,
$N_{\QQ(\zeta_{8p})/L'_1} (1-\zeta_{8p})$ is a unit in $L'_1$.

Note that
\begin{align*}
N_{L'_1/Q(\zeta_8)}N_{\QQ(\zeta_{8p})/L'_1} \Big ( 1-\zeta_{8p} \Big )
&= N_{\QQ(\zeta_{8p})/\QQ(\zeta_8)} \Big ( 1-\zeta_{8p} \Big  ) \\ &=
\prod_{\substack{1 \leq b < 8p \\ b \equiv 1 \text{ mod } 8 \\ (b,p)=1}} \Big ( 1-\zeta_{8p}^b \Big ) \\
&= \frac{\prod_{\substack{1 \leq b < 8p \\ b \equiv 1 \text{ mod } 8}} \Big ( 1-\zeta_{8p}^b \Big )}{\prod_{\substack{1 \leq b < 8p \\ b \equiv 1 \text{ mod } 8 \\ b \equiv 0 \text{ mod }p}} \Big ( 1-\zeta_{8p}^b \Big )} \\
\end{align*}
We also have
$$\prod_{\substack{1 \leq b < 8p \\ b \equiv 1\text{ mod } 8}} \Big ( 1-\zeta_{8p}^b \Big ) = 1-\zeta_8.$$
Since $p \equiv 3 \mmod 8$, the only value of $b$ between $1$ and $8p$ that is
$0 \mmod p$ and $1 \mmod 8$ is $3p$.
Therefore
\begin{align*}
N_{\QQ(\zeta_{8p})/\QQ(\zeta_{8})} \Big ( 1-\zeta_{8p} \Big )
= \frac{1-\zeta_8}{1-\zeta_8^3} =
\zeta_8^{-1} (\sqrt{2}-1)
\end{align*}
Since $L'_1/K_1$ is a degree $\ell \neq 2$ extension, $\sqrt{2}-1$ is the norm
of a unit in $L'_1$.

Now note that $L_1 \subseteq \QQ(\zeta_{8pm})$ for some $m$ with $\gcd(\ell p, m)=1$.
Since $1-\zeta_{8p}$ is a unit in $\QQ(\zeta_{8p})$, it is a norm for
the unramified extension $\QQ_p(\zeta_{8pm}) / \QQ_p(\zeta_{8p})$.
Since $\zeta_8^{-1} (\sqrt{2}-1)$ is the norm of $1-\zeta_{8p}$ for
$\QQ(\zeta_{8p})/\QQ(\zeta_8)$, $\zeta_8^{-1} (\sqrt{2}-1)$
is a norm for $\QQ_p(\zeta_{8pm}) / \QQ_p(\zeta_{8})$. This means that it is a
norm for the subextension obtained by completing $L_1/\QQ(\zeta_8)$ at a prime
above $p$. Since $\ell$ is odd, $\sqrt{2}-1$ is a local norm at all primes above $p$.
\end{proof}
Therefore if all ramified primes in $L_1/K_1$ lie over primes congruent to $3 \mmod 8$,
then the norm residue symbols of the units of $K_1$ are all trivial and so
by Theorem \ref{chevalley} we have $\rk A_1^\Delta = 1$.

The norm residue symbols are not necessarily trivial
at primes that are $7 \mmod 8$.
Recall that we let $t$ be the number of rational primes that ramify in $L/\QQ$
and that all such primes are inert in $K_0/\QQ$.
Let $0 \leq s \leq t$ be the number of rational
primes below the ramified primes in $L_1/K_1$ that are congruent to $7 \mmod 8$.
There are then $s$ independent norm residue symbols, all of which must be trivial in order for
the fundamental unit to be a global norm.
Therefore the probability that $\rk A_1^\Delta = 1$ is $\ell^{-s}$ and the probability that
$\rk A_1^\Delta = 0$ is $1-\ell^{-s}$.

\subsection{$L_n/K_n$ for $n>1$}
For the general $n>1$ case, there are $2^{n-1}$
independent elements of $U_n^{rel}$ (see Corollary \ref{relative_unit_fixed}).
By Theorem \ref{allornothing}, for $n \geq 2$, either all of these relative units in $K_n$ are norms
of elements in $L_n$ modulo $\ell^{th}$ powers or none of them are, and there are $2^{n-2}$
independent norm residue symbols that determine if the units are norms.

Under the assumption that norm residue symbols are equidistributed,
we can compute
the probabilities that the relative units modulo $\ell^{th}$ powers are norms.
The probability that the relative units in $K_1$ are all norms of elements in $L_1$
is $1/\ell^s$. For $n>1$, the probability that the relative units in $K_n$ are all
norms of elements in $L_n$ is $1/\ell^{2^{n-2}t}$.

Let $A_n$ be the $\ell$-class group of $L_n$.
By Theorem \ref{chevalley},
we have $$\rk A_n^\Delta = 2^n t - 1 -r_n$$ where $r_n$ is the
rank of the matrix of norm residue symbols for the units.
This assumes that $\ell$ does not divide $h(K_n)$, which is quite possibly true.
If $\ell$ does divide $h(K_n)$, then these heuristics apply to the class group of $L_n$ excluding the contribution from $K_n$.

The probability that none of the units are norms in $K_n$ is the product of
the probabilities that none of the relative units are norms in $K_j$ for $1 \leq j \leq n$.

If $r_n$ is maximal, then $r_n = 2^n-1$. Therefore
$$\rk A_n^\Delta = 2^n t - 2^n.$$
The probability
that the matrix is full rank at each step (which means $r_n$ is maximal) is
$$(1-1/\ell^s) \prod_{j=2}^{n} (1-1/\ell^{2^{j-2}t}).$$

In general, we can determine the rank of the matrix of norm residue symbols by considering
the matrices corresponding to the units of $K_{j+1}$, $0 \leq j \leq n-1$, independently.
At each level, the matrix of norm residue symbols is either full rank or identically zero.

For example, suppose we wish to know the probability
that the rank of the norm residue matrix is five.
Since there are $2^{n-1}$ relative units in $K_n$, $n>0$,
and either all relative units at a given level are norms or none of them are,
then the only way for the rank of the matrix to be five is for the unit in $K_1$ and
the four relative units of $K_3$ to be norms.

Consider the case where just one prime ramifies in $L/K$.
If the ramified prime is $7 \mmod 8$,
then a full rank norm residue symbol matrix implies that $\rk A_n^\Delta = 0$.
The probability that $\rk A_n^\Delta = 0$ is
\begin{align}\label{rank0equation}
(1-1/\ell) \prod_{j=2}^{n} (1-1/\ell^{2^{j-2}}).
\end{align}
So we expect that
\begin{align*}
\lim_{n \rightarrow \infty} (1-1/\ell) \prod_{j=2}^{n} (1-1/\ell^{2^{j-2}})
\end{align*}
of fields will have $A_n^\Delta \simeq 1$ up the tower.
We can then apply Proposition \ref{grasprop} to see that
if $A_n^\Delta \simeq 1$, then the
$\ell$-class group $A_n$ is also trivial up the tower.

Now consider the case where the ramified prime is $3 \mmod 8$.
Then by Proposition \ref{trivialsymbols}, the fundamental
unit of $K_1$ is always a norm, and so
the probability that
$\rk A_n^\Delta = 1$
is
\begin{align}\label{rank1equation}
\prod_{j=2}^{n} (1-1/\ell^{2^{j-2}}).
\end{align}
So we expect in the limit that
\begin{align*}
\lim_{n \rightarrow \infty} \prod_{j=2}^{n} (1-1/\ell^{2^{j-2}})
\end{align*}
of fields will have $A_n^\Delta  \simeq \ZZ/\ell\ZZ$
up the tower. Therefore the rank of $A_n^\Delta$ is bounded, and so the rank
of $A_n$ is bounded.
However, by the following proposition, this implies that the order of $A_n$ is also bounded.
\begin{proposition}[Washington \cite{washington1975class}, Proposition 1]
Let $K$ be a number field, let $p$ be a prime, and let $K_\infty/K$ be any
$\ZZ_p$-extension. $K_n$ is the unique subfield of $K_\infty$ that is degree $p^n$ over $K$. Let $\ell \neq p$ be any other prime number and let $e_n$ be such that
$\ell^{e_n}$ is the exact power of $\ell$ dividing $h(K_n)$. Let $A_n$ be the
$\ell$-class group of $K_n$. Then if $\rk A_n$ is bounded as $n \rightarrow \infty$,
then $e_n$ is also bounded.
\end{proposition}

\subsection{Washington's Conjecture}

In a 1975 paper, Washington made the following conjecture.
\begin{conjecture}[Washington \cite{washington1975class}]\label{washconj}
Let $\ell \neq p$ be primes.
Let $K$ be a number field. Let $K_\infty/K$ be a $\ZZ_p$-extension and let $K_n$ be the unique subfield of $K_\infty$ which is degree $p^n$ over $K$. Let $h(K_n) = \ell^{e_n}r$, $\ell \nmid r$ be the class number of $K_n$. Then there exist $\beta \geq 0$ and $\gamma$ independent of $n$ such that $e_n = \beta p^n + \gamma$ for sufficiently large $n$.
\end{conjecture}

By Lemma \ref{normprop}, we have an increasing sequence
$$ \ldots \subseteq N_{L_n/K_n} L_n^\times \cap E_{K_n} \subseteq  \hspace{.1in} N_{L_{n+1}/K_{n+1}} L_{n+1}^\times \cap E_{K_{n+1}} \subseteq  \ldots $$
For $n \geq 2$, the probability that there is a strict increase from step $n$ to step $n+1$ is
$$1/\ell^{2^{n-1}t}.$$
Therefore the probability that there is a strict increase infinitely often is
bounded above by
$$\sum_{j=N}^\infty 1/\ell^{2^{j-1}t}$$ for every $N$, which converges quickly to zero
as $N \rightarrow \infty$.
In other words, the order of $$N_{L_n/K_n} L_n \cap E_{K_n}$$ stabilizes with probability 1.
Therefore
by Chevalley's formula (Theorem \ref{chevalley}) we expect that
there exists an $N_0$ such that the $\ell$-rank of $A_n^\Delta$ is $$\beta \cdot 2^n + \gamma$$
for $n \geq N_0$. Therefore we expect $A_n^\Delta$ to satisfy Conjecture \ref{washconj}.

\section{Data}\label{datasec}

In this section, we include some computational results on the ranks of class groups
in the first few layers in cubic lifts of the anti-cyclotomic $\ZZ_2$-extension of $\QQ(i)$.
We restrict to the case where just one prime ramifies in $L/\QQ$.
The class groups were computed using Sage \cite{sagemath} assuming the Generalized
Riemann Hypothesis.

From \cite{hubbardwashington2017}*{Proposition 5}, we have explicit polynomials for the
first layers of the anti-cyclotomic $\ZZ_2$-extension of $\QQ(i)$.
The extension $K_1/K_0$ has defining polynomial $x^2+2$ and $K_2/K_0$ has
defining polynomial $x^4+2$.
For cubic extensions, we can only compute the class group of $L_n$ for $n \leq 2$
with our current computational resources.

Using the results from \S \ref{sec:app}, we compute the predicted probabilities
for the ranks of $A_n^\Delta$.

\begin{center}
\begin{table}[H]
\begin{tabular}{|l|llll|}
\hline
$n$ & $P(\rk A_n^\Delta = 0)$ & $P(\rk A_n^\Delta = 1)$ & $P(\rk A_n^\Delta = 2)$ & $P(\rk A_n^\Delta = 3)$  \\
\hline
0 & 1 & 0 & 0 & 0 \\
1 & $2/3$ & $1/3$ & 0 & 0 \\
2 & $4/9$ & $2/9$ & $2/9$ & $1/9$ \\
\hline
\end{tabular}
\caption{One prime congruent to $7 \mmod 8$ ramifies.}
\end{table}
\end{center}

\begin{center}
\begin{table}[H]
\begin{tabular}{|l|llll|}
\hline
$n$ & $P(\rk A_n^\Delta = 0)$ & $P(\rk A_n^\Delta = 1)$ & $P(\rk A_n^\Delta = 2)$ & $P(\rk A_n^\Delta = 3)$  \\
\hline
0 & 1 & 0 & 0 & 0 \\
1 & 0 & $1$ & 0 & 0 \\
2 & 0 & $2/3$ & 0 & $1/3$ \\
\hline
\end{tabular}
\caption{One prime congruent to $3 \mmod 8$ ramifies.}
\end{table}
\end{center}

Tables \ref{7mod8table} and \ref{3mod8table} give our computational results.
We include the number of fields with
class group of the given rank and the proportion in parentheses.
Recall that if $A_n^\Delta$ is trivial then so is $A_n$.

\begin{center}
\begin{table}[H]
\begin{tabular}{|l|l|l|l|l|l|}
\hline
$n$ & \# fields & Rank 0 & Rank 1 & Rank 2 & Rank 3     \\
\hline
0 & 156 & 156 (1) & 0 (0)  & 0 (0)  & 0 (0)  \\
1 & 29856 & 19899 (.6665) & 8821 (.2955) & 1136 (.0380) & 0 (0)  \\
2 & 25 & 11 (.44)  & 4 (0.16) & 9 (.36) & 1 (.04)  \\
\hline
\end{tabular}
\caption{Rank of the $3$-class group $A_n$ when one prime congruent to $7 \mmod 8$ ramifies in $L/\QQ$.}
\label{7mod8table}
\end{table}
\end{center}

\vspace{-.2in}

\begin{center}
\begin{table}[H]
\begin{tabular}{|l|l|l|l|l|l|l|l|}
\hline
$n$ & \# fields & Rank 0 & Rank 1 & Rank 2 & Rank 3 & Rank 4    \\
\hline
0 & 155 & 155 (1) & 0 (0)  &  0 (0) &  0 (0)&  0 (0)\\
1 & 873 & 0 (0)  & 589 (.6747) & 284 (.3253) & 0 (0) &  0 (0) \\
2 & 20 & 0 (0)  & 10 (0.5) & 5 (.25) & 2 (.1) & 3 (.15) \\
\hline
\end{tabular}
\caption{Rank of the $3$-class group $A_n$ when one prime congruent to $3 \mmod 8$ ramifies in $L/\QQ$.}
\label{3mod8table}
\end{table}
\end{center}

Heuristics derived from the theory of ambiguous ideals
and strongly ambiguous ideals explain the difference between
the rank of $A_n^\Delta$ and the rank of $A_n$. Although details are outside
the scope of this paper, they can be found in  \cite{kirschthesis}*{\S 6.2.2}.

For example, we can explain
why $2/3$ of the class groups of $A_1$ have rank $0$ when the ramified prime is $7 \mmod 8$
and why $2/3$ of the class groups of $A_1$ have rank 1 when the ramified prime is
$3 \mmod 8$. The heuristics depend on the behavior of the ambiguous and strongly ambiguous ideals.
The fundamental difference between the $3 \mmod 8$ case and the $7 \mmod 8$ case can be
explained via the following proposition.
\begin{proposition}[Lemmermeyer \cite{lemmermeyer2013ambiguous2}, Theorem 1]
Let $\bar{A}_n^\Delta$ be the strongly ambiguous class group of $L_n$.
Then $$[A_n^\Delta : \bar{A}_n^\Delta] = [E_{K_n} : N_{L_n/K_n} E_{L_n}] \big / [E_{K_n} : E_{K_n} \cap N_{L_n/K_n} L_n^\times].$$
\end{proposition}
From the proof of Proposition \ref{trivialsymbols}, we can see that for primes above $3 \mmod 8$,
the fundamental unit of $K_1$ is the norm of a unit of $L_1$. Therefore in this case all
ambiguous ideals are strongly ambiguous. However in the $7 \mmod 8$ we may have
ambiguous ideals that are not strongly ambiguous.
To develop the heuristics, we consider the matrices of Artin symbols
(as in Wittmann \cite{wittmannthesis}). These matrices have a different form depending on
if there are ambiguous ideals that are not strongly ambiguous, and this form is what
determines the predicted probabilities of class group rank.

Note that there are more class groups computed for the $n=1$ case when
$p \equiv 7 \mmod 8$. This is because
the class group in this case has the same rank as the class group of the
real subfield of $L_1$, and we can compute class group of larger discriminants
when the degree is lower.

\begin{proposition}
Let $K_\infty/K_0$ be the anti-cyclotomic $\ZZ_2$-extension of $K_0 = \QQ(i)$.
Let $L/\QQ$ be a cyclic degree $\ell$ extension in which only primes that are congruent to $7 \mmod 8$
ramify. Let $L_1/K_1$ be the lift of $L/\QQ$
and $\hat{L}_1$ be the real subfield of $L_1$.
Let $A(L_1)$ be the $\ell$-class group of $L_1$
and let $A(\hat{L}_1)$ be the $\ell$-class group of $\hat{L}_1$. Then $$A(\hat{L}_1) \simeq A(L_1).$$
\end{proposition}
\begin{proof}
We have the following diagram of fields:
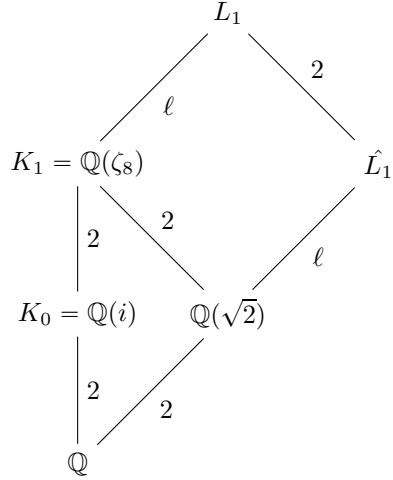
\begin{center}
\begin{tikzpicture}[node distance = 2.5cm, auto]
    \node (Q) at (2,0) {$\QQ$};
    \node (K) at (2,2) {$K_0 = \QQ(i)$};
    \node (K2) at (4,2) {$\QQ(\sqrt{2})$};
    \node (Kw) at (6,4) {$\hat{L_1}$};
    \node (L) at (2,4) {$K_1 = \QQ(\zeta_8)$};
    \node (Lw) at (4,6) {${L_1}$};
    \draw[-] (Q) to node [swap] {2} (K);
    \draw[-] (K2) to node [swap] {$\ell$} (Kw);
    \draw[-] (Kw) to node [swap] {2} (Lw);
    \draw[-] (K) to node [swap] {2} (L);
    \draw[-] (L) to node [swap] {$\ell$} (Lw);
    \draw[-] (Q) to node [swap] {2} (K2);
    \draw[-] (K2) to node [swap] {2} (L);
\end{tikzpicture}
\end{center}
Let $\Delta = \Gal(L_1/K_1)$ and let $\hat{\Delta} = \Gal \big (\hat{L}_1/\QQ(\sqrt{2}) \big)$.
Let
$$[E_{K_1}:E_{K_1} \cap N_{L_1/K_1}(L_1^\times)] = \ell^e$$
and
$$[E_{\QQ(\sqrt{2})}:E_{\QQ(\sqrt{2})} \cap N_{\hat{L}_1/\QQ(\sqrt{2})}(\hat{L}_1^\times)] = \ell^{\hat{e}}.$$
Then by
Chevalley's formula for $\hat{L_1}/\QQ(\sqrt{2})$,
$$| \ClLp(\hat{L_1})^\Delta | = \ell^{t-1-\hat{e}}$$
and similarly for $L_1/\QQ(\zeta_8)$,
$$| \ClLp({L_1})^{\hat{\Delta}} | = \ell^{t-1-e}.$$
Since there is only one fundamental unit in $\QQ(\zeta_8)$ and it is $\sqrt{2}-1$,
and since $\sqrt{2}-1$ is a norm from $L_1$ if and only if it is a norm from $\hat{L}_1$,
we have $e = \hat{e}$. Therefore $| \ClLp(\hat{L_1})^{\hat{\Delta}}  | = | \ClLp({L_1})^\Delta |$.
Since these are both elementary groups, because $\QQ(\sqrt{2})$ and $\QQ(\zeta_8)$
both have class number 1, they must be isomorphic.

Using the same notation as Washington \cite{washington2012introduction}*{\S 10.2}, let
$$\Gal(L_1/\hat{L_1}) = \{1,J\}$$ where $J$ is complex conjugation and let
$$\ClLp({L_1})^\pm = \frac{1\pm J}{2} \ClLp({L_1}).$$
Then
$$\ClLp({L_1}) = \ClLp({L_1})^+ \oplus \ClLp({L_1})^-.$$

Then since $\ClLp(\hat{L_1})$ injects into $\ClLp({L_1})^+$, we have the following sequence of
embeddings:
$$ \ClLp(\hat{L_1})^{\hat{\Delta}} \hookrightarrow \left(\ClLp({L_1})^+\right)^\Delta \hookrightarrow
\ClLp({L_1})^\Delta \simeq \ClLp(\hat{L_1})^{\hat{\Delta}}$$
and therefore we must have equality. Therefore
$$\left(\ClLp({L_1})^+\right)^\Delta \simeq \ClLp({L_1})^\Delta$$
and so $\left(\ClLp({L_1})^-\right)^\Delta$ must be trivial. By Nakayama's lemma, $\ClLp({L_1})^-$
must also be trivial.
Additionally,
$$A(L_1)^+ = N_{L_1/\hat{L}_1} A(L_1) \hookrightarrow A(\hat{L}_1) \hookrightarrow A(L_1)^+$$
and therefore
$$A(L_1) \simeq A(L_1)^+ \simeq A(\hat{L}_1).$$
\end{proof}

\subsection{The Structure Theorem}

During the course of this research, we encountered a numerical example that
exhibited a phenomenom for the $\ell \neq p$ case that could not have occured in the $\ell=p$
case. We wondered if there could be a uniform algebraic treatment of the cases for all $\ell$
and if there were some undiscovered structure theorem for the $\ell \neq p$ case. The
following indicates that this might not be the case.

First, we review the structure theorem for the $\ell$-class group in $\ZZ_\ell$-extensions.
\begin{theorem}
Let $L_\infty/L$ be a $\ZZ_\ell$-extension. Let $\ell^{e_n}$ be the exact power of $\ell$
dividing the class number of $L_n$. Then there exist integers $\lambda \geq 0$, $\mu \geq 0$, and $\nu$, all independent of $n$, and an integer $n_0$ such that for all $n \geq n_0$,
$$e_n = \mu \ell^n + \lambda n + \nu.$$
\end{theorem}

Let $K_0$ be an imaginary quadratic field and let $L/\QQ$ be an extension of degree $\ell$.
Let $K_\infty/K_0$ be the anti-cyclotomic $\ZZ_\ell$-extension and let $L_\infty/L_0$ be
its lift.
Furthermore, let
$$\Lambda = \ZZ_\ell[[T]]$$ and
$$\nu_n = (1+T)^{\ell^n-1} + (1+T)^{\ell^n-2} + \ldots + (1+T) + 1.$$
Then an elementary $\Lambda$-module $E$ is defined to be one of the form
$$E = \bigoplus_{i} \Lambda/(\ell^{\mu_i}) \oplus \bigoplus_j \Lambda/(f_j)$$
where $\mu_i > 0$ is an integer and $f_j$ is a distinguished polynomial, which
means it is monic and $\ell$ divides each coefficient (except the leading monic coefficient).
Let $A_n$ be the $\ell$-class group of $L_n$ and $\ell^{e_n}$ be the exact power of $\ell$
dividing the class number of $L_n$.

\begin{proposition}[{\cite{hubbardwashington2017}*{Proposition 12}}]\label{prop12}
There exist an elementary $\Lambda$-module $E$ and a finite $\Lambda$-module
$F$ such that
$$|A_n| = |A_0| \cdot |F/\nu_n F| \cdot |E/\nu_n E| \geq \ell^{e_0 + \mu(\ell^n-1)}.$$
In particular,
$$\mu \leq \frac{e_n-e_0}{\ell^n-1}.$$
\end{proposition}

Hubbard and Washington prove the following theorem.
\begin{theorem}[{\cite{hubbardwashington2017}*{Theorem 2}}]\label{theorem2}
Suppose $s$ distinct primes $q \neq \ell$ are inert in $K_0/\QQ$ and
ramify in $L_0/K_0$, a degree $\ell$ extension. Then $\mu \geq s-1$ for the $\ZZ_\ell$-extension
$L_\infty/L_0$.
\end{theorem}

Consider the class groups in the $\ZZ_2$-extension (so $\ell=2$).
Assume two primes ramify in $L_0/K_0$ and we have
$e_0 = 1$ and $e_1 = 2$.
Then by Proposition \ref{prop12}
$$\mu \leq \frac{e_1-e_0}{\ell^1-1} = 1$$
and by Theorem \ref{theorem2},
$$\mu \geq 1.$$
Therefore $\mu = 1$.

By \cite{hubbardwashington2017}*{Proposition 17}, there exists a $\Lambda$-module $E'$
such that
$$\ell^{e_n-e_0 - \mu(\ell^n-1)} = |F/\nu_n F| \times |E'/\nu_n E'|.$$
For $n=1$, the left-hand side is 1, and therefore both orders on the right must be 1.
Then by Nakayama's Lemma, both $F$ and $E'$ must be trivial.

Therefore $\ell^{e_n-e_0 - \mu(\ell^n-1)} = 1$ for all $n$. Since $e_0 = 1$, $\mu=1$ and
$\ell=2$, we have
\begin{align}e_n = 2^n.\end{align}\label{eneqn}

Now let's consider the $\ell \neq p$ situation. Let $K_0 = \QQ(i)$ and consider the
cyclic cubic extension $L_0/K_0$ given by
$$x^3 - 76x^2 + 1636x - 7064$$
in which only primes over 7 and 31 ramify.

We have the following 3-class groups.
\begin{align*}
A_{0} &= \ZZ/3\ZZ \\
A_{1} &= (\ZZ/3\ZZ)^2 \\
A_{2} &= (\ZZ/3\ZZ)^6 \\
\end{align*}
By Chevalley's formula, $e_n \geq 2^n-1$, so
we are in the analogue of the $\mu \geq 1$ situation.
But using the results above for the $\ell=p$ case, the sequence
$e_0=1$, $e_1=2$, $e_2 = 6$ is not possible for the $\ell$-part of class group (see Equation \ref{eneqn}).
So the theory resulting from the structure theorem
for the $\ell$-class group in $\ZZ_\ell$-extensions does not extend to the $p$-class group
in $\ZZ_\ell$-extensions when $p \neq \ell$.

\pagebreak

\bibliographystyle{amsplain}
\bibliography{akirsch}

\end{document}